\documentclass[letterpaper, 10pt, twocolumn]{ieeeconf}      

\IEEEoverridecommandlockouts                              
\def\BibTeX{{\rm B\kern-.05em{\sc i\kern-.025em b}\kern-.08em
    T\kern-.1667em\lower.7ex\hbox{E}\kern-.125emX}}
    
\usepackage[left=54pt, right=54pt,bottom=54pt, top=54pt]{geometry}

\usepackage{amsmath,mathrsfs,amsfonts,amssymb,graphicx,epsfig,stmaryrd}
 \usepackage{amsthm}
\usepackage{subcaption}
\usepackage{caption}
\usepackage{color,multirow,rotating}
\usepackage{algorithm,algpseudocode,algorithmicx}
\usepackage{cite,url,framed,bm,balance,dsfont,varwidth}
\usepackage{nicefrac,stmaryrd}

\newtheorem{theorem}{Theorem}

\newtheorem{remark}{Remark}

\newtheorem{example}{Example}
\usepackage{cleveref}
\usepackage{tikz}
\usetikzlibrary{calc,trees,positioning,arrows,chains,shapes.geometric,decorations.pathreplacing,decorations.pathmorphing,shapes, matrix,shapes.symbols}

\newcommand{\differential}{{\rm{d}}}
\renewcommand{\det}{{\mathrm{det}}}

\newcommand{\complexi}{{\mathrm{i}}}

\newcommand{\RNum}[1]{\uppercase\expandafter{\romannumeral #1\relax}}

\allowdisplaybreaks

\title{\LARGE\textbf{Weyl Calculus and Exactly Solvable Schr\"{o}dinger Bridges\\with Quadratic State Cost}
}

\author{Alexis M.H. Teter, Wenqing Wang, Abhishek Halder
\thanks{Alexis M.H. Teter is with the Department of Applied Mathematics, University of California, Santa Cruz, CA 95064, USA,
       {\tt\small{amteter@ucsc.edu}}.\\
        Wenqing Wang and Abhishek Halder are with the Department of Aerospace Engineering, Iowa State University, Ames, IA 50011, USA,
        {\tt\small{\{wqwang,ahalder\}@iastate.edu}}.\\
        This work was supported in part by NSF grants 2112755 and 2111688.
}}

\begin{document}

\maketitle

\bstctlcite{IEEEexample:BSTcontrol} 

\begin{abstract}
Schr\"{o}dinger bridge--a stochastic dynamical generalization of optimal mass transport--exhibits a learning-control duality. Viewed as a stochastic control problem, the Schr\"{o}dinger bridge finds an optimal control policy that steers a given joint state statistics to another while minimizing the total control effort subject to controlled diffusion and deadline constraints. Viewed as a stochastic learning problem, the Schr\"{o}dinger bridge finds the most-likely distribution-valued trajectory connecting endpoint distributional observations, i.e., solves the two point boundary-constrained maximum likelihood problem over the manifold of probability distributions. Recent works have shown that solving the Schr\"{o}dinger bridge problem with state cost requires finding the Markov kernel associated with a linear reaction-diffusion PDE where the state cost appears as a state-dependent reaction rate. We explain how ideas from Weyl calculus in quantum mechanics, specifically the Weyl operator and the Weyl symbol, can help determine such Markov kernels. We illustrate these ideas by explicitly finding the Markov kernel for the case of convex quadratic state cost via Weyl calculus, recovering as well as generalizing our earlier results but avoiding tedious computation with Hermite polynomials.   
\end{abstract}


\section{Introduction}\label{sec:introduction}

The purpose of this work is to highlight the usefulness of Weyl calculus tools in computing the kernels or Green's functions (a.k.a. propagators \cite[Ch. 2.6]{sakurai2020modern} in quantum mechanics) for linear PDE initial value problems (IVPs) in diffusion and control.

The mathematical development of Weyl calculus originated in quantum mechanics, but its relevance in solving problems outside quantum mechanics, especially in diffusion and control problems, appears to be less known.

Common references \cite{cohen2012weyl}, \cite[Ch. 2-3]{robert2021coherent}, \cite{derezinski1993some,derezinski2021introduction,fedosov1996deformation} on Weyl calculus appeal to quantum physics and field-theoretic ideas to introduce the core mathematical concepts, albeit for good reasons. However, this style of exposition, even when mathematically rigorous, can become a barrier to the broader systems-control audience. Most researchers in systems-control, barring the cognoscenti, may not be able to afford the significant time and resource to transfer the mathematical ideas of Weyl calculus to their specific applications, or are simply not fortunate enough to be talking with the right person at the right time. 

This work takes a different expository approach: a purposefully less rigorous recipe-style introduction of selective Weyl calculus tools with illustrative calculations. The recipe for finding the kernel that we exemplify, follows the path:
$${\text{PDE}} \:\longrightarrow \:{\text{Weyl operator}} \:\longrightarrow \:{\text{Weyl symbol}} \:\longrightarrow \:{\text{kernel}}.$$
We explain the two middle ingredients: Weyl operator and symbol, and their connections with the familiar ends: PDE and kernel.

The technical motivation for this work came from solving the Schr\"{o}dinger bridge (SB) problems -- a class of stochastic optimal control problems with deadline, controlled diffusion, and endpoint distribution constraints. Such problems and their solutions are finding use in stochastic control \cite{chen2016entropic,chen2021stochastic,caluya2021wasserstein,caluya2021reflected,pavon2021data,haddad2020prediction,nodozi2022schrodinger, vargas2023bayesian,10347388} and in diffusion models in artificial intelligence (AI) \cite{de2021diffusion,liu2023i2sb,peluchetti2023diffusion,yang2023diffusion,shi2024diffusion}. The SB problems exhibit a learning-control duality. Viewed as a stochastic control problem, the SB finds an optimal control policy that steers a given joint state statistics to another while minimizing the total control effort subject to controlled diffusion and deadline constraints. Viewed as a stochastic learning problem, the SB solves the two point boundary-constrained maximum likelihood problem over the manifold of probability distributions. In fact, the latter viewpoint led to the original formulation of the SB problem by Erwin Schr\"{o}dinger in 1931-32 \cite{schrodinger1931umkehrung,schrodinger1932theorie}.

Numerically solving the SB problems with state costs for generic endpoint distributional problem data, require finding the Markov kernels of certain linear reaction-diffusion PDEs where the state costs play the role of (nonlinear) reaction rates\footnote{For brevity, we will not detail the SB formulation and the known derivation of the reaction-diffusion PDEs using the first order conditions of optimality followed by the Hopf-Cole transform \cite{hopf1950partial,cole1951quasi}. Here, these details are not needed and they will distract us from the main points of this work. We refer the readers to \cite[p. 274-276]{chen2021stochastic}, \cite[Sec. 3]{teter2024schr} for these details.}. In such a setting, finding the kernel becomes a nontrivial task. While laborious specialized argument as in \cite{teter2024schr} may work in certain cases, a principled approach remains desired. As we explain in this work, the Weyl calculus tools can serve as a promising alternative.

We mention here that despite the authors' specific motivation of solving SB problems with state cost, the scope of the Weyl calculus tools are broad. They should be of interest to the systems-control community for finding or analyzing linear PDE IVP solutions of interest at large. 

\subsubsection*{Contributions}
\begin{itemize}
    \item Our primary contribution is pedagogical in that we explain the general steps to derive the Green's functions or kernels of certain linear PDE IVPs using Weyl calculus tools in a systematic manner. 

    \item Our secondary contribution is to analytically recover the kernel for the SB problem with quadratic state cost using the Weyl calculus tools.  
\end{itemize}

\subsubsection*{Organization} In Sec. \ref{sec:NotationsAndBackground}, we fix some notations and background concepts that find use in the following development. The relevant Weyl calculus ideas and computational steps are explained in Sec. \ref{sec:WeylCalculus}. In Sec. \ref{sec:ClassicalSBFixedQ}, we put these ideas and tools in action to analytically compute the kernel for a reaction-diffusion PDE with quadratic state cost that appeared in the Schr\"{o}dinger bridge literature. Concluding remarks in Sec. \ref{sec:conclusions} close the paper. 


\section{Notations and Background}\label{sec:NotationsAndBackground}
\noindent\textbf{Poisson bracket.} Let ${\rm{Skew}} \left(2n,\mathbb{R}\right)$ denote the set of $2n\times 2n$ skew-symmetric matrices with real entries. Given sufficiently smooth functions $f,g:\mathbb{R}^{2n}\mapsto\mathbb{R}$, define their \emph{Poisson bracket} $\{\cdot,\cdot\}$ as
\begin{align}
\{f,g\} := \langle\nabla f, \bm{\Omega}\nabla g\rangle, \: \bm{\Omega} := \begin{pmatrix}
\bm{0} & \bm{I}_{n}\\
-\bm{I}_{n} & \bm{0}
\end{pmatrix}\in{\rm{Skew}}\left(2n,\mathbb{R}\right),
\label{defPoissonBracket}  
\end{align}
where $\langle\cdot,\cdot\rangle$ denotes the standard Euclidean inner product. Note that $\bm{\Omega}^{-1}=\bm{\Omega}^{\top} = -\bm{\Omega}$, and $\det(\bm{\Omega})=1$. As is well-known \cite[Prop. 2.23]{hall2013quantum}, the Poisson bracket is bilinear, skew symmetric, and satisfies the Jacobi identity.

For $\bm{x},\bm{\xi}\in\mathbb{R}^{n}$, we think of $f,g$ as mappings which act on the tuple $(\bm{x},\bm{\xi})$, i.e., $f,g : (\bm{x},\bm{\xi})\in\mathbb{R}^{2n}\mapsto\mathbb{R}$. Following \cite[p. 128-9]{taylor1984noncommutative}, let $$\{f,g\}_{1}(\bm{x},\bm{\xi}):= \frac{1}{2\complexi}\{f,g\}(\bm{x},\bm{\xi}), \quad \complexi := \sqrt{-1},$$ where $\{f,g\}$ is the Poisson bracket. More generally, for any $j\in\mathbb{N}_{0}:=\{0,1,2,\hdots\}$, let
\begin{align}
     \{f,g\}_{j}\left(\bm{x},\bm{\xi}\right):=&\left( \frac{1}{2\textrm{i}}\right)^{j} \left( \sum_{k=1}^{n} \left( \frac{\partial^2}{\partial y_k \partial \xi_k} - \frac{\partial^2}{\partial x_k \partial \nu_k}\right)\right)^{j}\nonumber\\
     &\qquad\qquad f(\bm{x}, \bm{\xi})g(\bm{y}, \bm{\eta})\big\vert_{\bm{y} = \bm{x}, \bm{\eta} =\bm{\xi}}.
     \label{BracketDef}
\end{align}
In particular, for $j=0$, definition \eqref{BracketDef} gives $$\{f,g\}_{0}\left(\bm{x},\bm{\xi}\right) = f\left(\bm{x},\bm{\xi}\right)g\left(\bm{x},\bm{\xi}\right),$$
i.e., the product of the two functions.

\noindent\textbf{Standard symplectic form.} For $\bm{x},\bm{y},\bm{\xi},\bm{\eta}\in\mathbb{R}^{n}$, the standard symplectic form associated with the matrix $\bm{\Omega}$ in \eqref{defPoissonBracket} is a bilinear mapping\footnote{More generally, the mapping $\sigma:\mathbb{C}^{2n}\times\mathbb{C}^{2n}\mapsto\mathbb{R}$ remains well-defined as stated here but we will not need this fact.} $\sigma:\mathbb{R}^{2n}\times \mathbb{R}^{2n} \mapsto \mathbb{R}$, given by 
\begin{align}
\sigma\left(\left(\bm{x},\bm{\xi}\right);\left(\bm{y},\bm{\eta}\right)\right) := \sum_{k=1}^{n}\left(\xi_{k}y_{k} - x_{k}\eta_{k}\right).
\label{defSymplecticForm}
\end{align}
The symplectic form $\sigma$ is skew-symmetric $(\sigma(\bm{v};\bm{w}) = \bm{v}^{\top}\bm{\Omega}\bm{w} = -\sigma(\bm{w};\bm{v})\:\forall \bm{v},\bm{w}\in\mathbb{R}^{2n})$, and non-degenerate (for any $\bm{v}\in\mathbb{R}^{2n}$, $\sigma(\bm{v};\bm{w}) = 0\:\forall\bm{w}\in\mathbb{R}^{2n}\:\Rightarrow\: \bm{v}=\bm{0}$).

\section{Weyl Calculus}\label{sec:WeylCalculus}
In this Section, we explain how Weyl calculus can be used to determine the Markov kernel associated with a linear PDE IVP of interest. While we are primarily motivated for the case when the PDE is reaction-diffusion type, we proceed formally in this Section in terms of operators. Along the way, we present necessary ideas and nomenclatures from Weyl calculus in an accessible manner.

For $t_{0}\geq 0$ and $\mathcal{X}\subseteq\mathbb{R}^{n}$, we begin by considering a linear PDE IVP 
\begin{align}
\frac{\partial}{\partial t}\widehat{\varphi} = -\mathcal{L} \widehat{\varphi}, \quad \widehat{\varphi}\left(t=t_{0},\bm{x}\right)=\widehat{\varphi}_{0}(\bm{x}) \;\text{(given)},
\label{PDE_form}
\end{align}
where the unknown $\widehat{\varphi}:[t_{0},\infty)\times\mathcal{X}\mapsto\mathbb{R}$, and the spatial operator $\mathcal{L}$ is time-independent. Formally, the solution of \eqref{PDE_form} can be expressed via the semigroup $\exp{\left(-(t-t_0)\mathcal{L}\right)}$ as
\begin{align}
    \widehat{\varphi} = \exp{\left(-\left(t-t_0\right)\mathcal{L}\right)} \widehat{\varphi}_{0}.
\label{PDEsol}    
\end{align}

We seek a kernel representation of the solution $\widehat{\varphi}$, i.e., for $0\leq t_{0} \leq s \leq t < \infty$, we seek to explicitly determine $\kappa(s, \bm{x}, t, \bm{y})$ referred to as the \emph{kernel} with suitable domain of definition, such that \eqref{PDEsol} is expressible as
\begin{align}
    \widehat{\varphi}(t,\bm{x}) = \int_{\mathcal{X}} \kappa(t_0, \bm{x}, t, \bm{y}) \widehat{\varphi}_{0}\left(\boldsymbol{y}\right) \mathrm{d} \boldsymbol{y}.
    \label{kernel_form}
\end{align}
The initial condition $\widehat{\varphi}\left(t=t_{0},\bm{x}\right)=\widehat{\varphi}_{0}(\bm{x})$ necessitates that 
\begin{align}
\kappa(t_0, \bm{x}, t_0, \bm{y}) = \delta\left(\bm{x}-\bm{y}\right),\,\text{the Dirac delta}.
\label{KernelIC}    
\end{align}
In what follows, we fix $\mathcal{X}\equiv\mathbb{R}^{n}$ for ease of exposition; the development goes through for general subsets of $\mathbb{R}^{n}$.

When the PDE in \eqref{PDE_form} is the forward Kolmogorov a.k.a. Fokker-Planck PDE, then $\mathcal{L}$ is an advection-diffusion operator and under mild regularity assumptions on the underlying drift and diffusion coefficients, the kernel $\kappa$ is a transition probability density; see e.g., \cite[Ch. 1]{stroock2008partial}. However, if $\mathcal{L}$ has an additional reaction term, then \eqref{PDE_form} is non-conservative and $\kappa$ no longer has the interpretation of transition probability density. In such more generic settings, an analytical handle on $\kappa$ can still help to numerically compute the solution of \eqref{PDE_form} via \eqref{kernel_form}. In the absence of an analytic handle on $\kappa$, numerically computing $\widehat{\varphi}$ becomes particularly non-trivial in the presence of (state dependent) reaction \cite[Sec. 5.2, 5.3]{teter2024solution}.  

\begin{example}[\textbf{Reaction-diffusion PDE IVP}] 
An instance of \eqref{PDE_form} is the reaction-diffusion PDE IVP with $\mathcal{L}\equiv -\Delta_{\bm{x}} + q(\bm{x})$, wherein $\Delta_{\bm{x}}$ denotes the Euclidean Laplacian operator, and $q(\cdot)$ is a given bounded continuous reaction rate. If $q$ is constant, then a change of variable $\widehat{\varphi} \mapsto \widehat{\vartheta} := e^{tq}\widehat{\varphi}$ transforms the problem to the heat PDE IVP:
$$\frac{\partial}{\partial t}\widehat{\vartheta} = \Delta_{\bm{x}}\widehat{\vartheta}, \quad \widehat{\vartheta}\left(t=t_{0},\bm{x}\right)=\widehat{\varphi}_{0}(\bm{x}),$$
for which $\kappa$ is well-known \cite[p. 44-47]{evans2022partial}. However, for state-dependent $q(\cdot)$, new ideas are needed. In Sec. \ref{sec:ClassicalSBFixedQ}, we will find $\kappa$ for quadratic $q(\cdot)$ using Weyl calculus tools discussed next.
\end{example}

\subsection{From PDE to Weyl Operator}\label{subsec:WeylOperator}
In pursuit of getting analytic handle on $\kappa$ in \eqref{kernel_form}, we start by defining the following operators as per the quantum mechanics convention \cite[p. 1]{cohen2012weyl}:
\begin{subequations}
\begin{align}
    X_k &:= x_k &\forall k\in[n], \label{defXk}\\
    D_k &:= \frac{1}{\textrm{i}} \frac{\partial}{\partial x_{k}}&\forall k\in[n],
\end{align}
\label{OperatorDefs}
\end{subequations}
and let $$\bm{X}:= \begin{pmatrix}X_1&\hdots& X_{n}\end{pmatrix}^{\!\top}, \quad \bm{D}:= \begin{pmatrix}D_1&\hdots& D_{n}\end{pmatrix}^{\!\top}.$$
We then determine a representation of $\exp{\left(-(t-t_0)\mathcal{L}\right)}$ in terms of the operators $\bm{X}$ and $\bm{D}$. This representation is referred to as the \emph{Weyl Operator} $H(\bm{X}, \bm{D})$. 

\begin{example}[\textbf{Weyl operator for the heat PDE}]\label{example:WeylOperatorHeatPDE}
Let \eqref{PDE_form} be the heat PDE IVP with $\mathcal{L}\equiv -\Delta_{\bm{x}}$. Letting $\nabla_{\bm{x}} := \begin{pmatrix} \frac{\partial}{\partial x_1} & \hdots & \frac{\partial}{\partial x_n} \end{pmatrix}^{\!\top}$, note that $$\vert \bm{D}\vert^2 := \langle \bm{D},\bm{D}\rangle = \left(-\complexi\right)^2\left(\nabla_{\bm{x}}\cdot\nabla_{\bm{x}}\right) = -\Delta_{\bm{x}},$$
i.e., $\Delta_{\bm{x}} = -\vert \bm{D}\vert^2$. So the corresponding Weyl operator is 
\begin{align*}
    H_{\mathrm{heat}}\left(\bm{X}, \bm{D}\right) = \exp{\left(-(t-t_0)\vert \bm{D}\vert^2\right)}.
\end{align*}
\end{example}

\subsection{From Weyl Operator to Weyl Symbol}\label{subsec:WeylSymbol}
After identifying the Weyl operator associated with our PDE of interest, we seek the corresponding \emph{Weyl symbol} denoted as $h(\bm{x}, \bm{\xi})$. 

Note that while the Weyl operator depends on $\bm{X}$ and $\bm{D}$, the Weyl symbol is dependent on variables $\bm{x}$ and $\bm{\xi}$. An approach for determining $h(\bm{x}, \bm{\xi})$ discussed in \cite[p. 32-36]{cohen2012weyl} proceeds as follows:
\begin{enumerate}
\item Rewrite the Weyl operator $H(\bm{X},\bm{D})$ such that all of the $D_k$ operators are multiplied on the right of the $X_k$ operators. As the operators in general are non-commutative, this may require the use of the commutation relation
\begin{align*}
    [X_k, D_k] := X_kD_k - D_kX_k = \textrm{i}, \quad k\in [n].
\end{align*}
\item Define $R(\bm{x}, \bm{\xi})$ by replacing the operator $\bm{X}$ with $\bm{x}$ and the operator $\bm{D}$ with $\bm{\xi}$ in the version of the Weyl operator $H(\bm{X}, \bm{D})$ satisfying Step 1 above.
\item Calculate the Weyl symbol $h(\bm{x}, \bm{\xi})$ via one of the following formulas:
\begin{subequations}
\begin{align}
    &h(\bm{x}, \bm{\xi}) \nonumber\\
    &= \frac{1}{\pi^{n}} \!\int_{\mathbb{R}^{2n}} \!\!R(\tilde{\bm{x}}, \tilde{\bm{\xi}}) \exp{\left(2\textrm{i}\langle\tilde{\bm{x}}-\bm{x},\tilde{\bm{\xi}}-\bm{\xi}\rangle\right)} \differential\tilde{\bm{x}} \:\differential\tilde{\bm{\xi}},\label{hIntegralRepresentation}\\
    &=  \sum_{m=0}^{\infty} \frac{1}{m!} \left( \frac{\complexi}{2}\right)^m \left( \frac{\partial^{m}}{\partial \bm{x}^m} \cdot \frac{\partial^{m}}{\partial \bm{\xi}^m} \right) R(\bm{x}, \bm{\xi}). 
    \label{hSeriesRepresentation}
\end{align}
\label{WeylSymbolRepresentationFormula} 
\end{subequations}
The differential operators in \eqref{hSeriesRepresentation} are understood as the $m$th order mixed partial derivatives.
In particular, \eqref{hSeriesRepresentation} can be seen as the series expansion of
\begin{align*}
    \exp{\left( - \frac{1}{2\textrm{i}} \frac{\partial}{\partial \bm{x}} \cdot \frac{\partial}{\partial \bm{\xi}}\right)} R(\bm{x}, \bm{\xi}). 
\end{align*}
\end{enumerate}
\begin{example}[\textbf{Weyl symbol for the heat PDE}]\label{example:WeylSymbolHeatPDE}
The Weyl operator $H_{\mathrm{heat}}(\bm{X}, \bm{D})$ found in Example \ref{example:WeylOperatorHeatPDE}, does not involve any multiplication of $\bm{X}$ with $\bm{D}$. Therefore, Steps 1 and 2 mentioned above are reduced to a direct replacement of $\bm{D}$ with $\bm{\xi}$. Because the resulting $R_{\mathrm{heat}}(\bm{x}, \bm{\xi}) = R_{\mathrm{heat}}(\bm{\xi}) = \exp{\left( -(t-t_0)\vert \bm{\xi}\vert^2\right)}$ is independent of $\bm{x}$, so \eqref{hSeriesRepresentation} immediately yields the Weyl symbol
\begin{align}
    h_{\mathrm{heat}}(\bm{x}, \bm{\xi}) &= e^{-(t-t_0)\vert \bm{\xi} \vert^2}.
    \label{SymbolHeat}
\end{align}
This example illustrates that finding the Weyl symbol associated with \eqref{PDE_form} is straightforward when the relevant Weyl operator depends on either $\bm{X}$ or $\bm{D}$ alone.
\end{example}

\subsection{From Weyl Symbol to Kernel}\label{subsec:Kernel}
We follow the development in H\"{o}rmander \cite[p. 161]{hormander2018}, to obtain the kernel $\kappa$ in \eqref{kernel_form} from the associated Weyl symbol $h\left(\bm{x},\bm{\xi}\right)$ as
\begin{align}
    \kappa(t_0, \bm{x}, t, \bm{y})\!=\! \frac{1}{(2\pi)^n} \!\int_{\mathbb{R}^{n}} \!\!h\!\left(\frac{\bm{x} + \bm{y}}{2}, \bm{\xi}\!\right)e^{ \textrm{i}\langle\bm{x}-\bm{y},\bm{\xi}\rangle}\differential\bm{\xi}.
    \label{SymbolToKernel}
\end{align}
The RHS of \eqref{SymbolToKernel} can be seen as the non-unitary inverse Fourier transform of $\bm{\xi}\mapsto h\left(\left(\bm{x}+\bm{y}\right)/2,\bm{\xi}\right)$. Thus, for \eqref{KernelIC} to hold, the Weyl symbol $h$ evaluated at $t_{0}$ must be equal to unity (see e.g., the heat symbol \eqref{SymbolHeat}). We will make use of this observation in our computation in Sec. \ref{subsec:solveWeylSymbolPDE}.

\begin{example}[\textbf{Kernel for the heat PDE}]\label{example:KernelHeatPDE}
Applying \eqref{SymbolToKernel} to the Weyl symbol $h\equiv h_{\rm{heat}}$ derived in \eqref{SymbolHeat}, we obtain
\begin{align}
    &\kappa_{\mathrm{heat}}(t_0, \bm{x}, t, \bm{y})\nonumber\\
    &=\frac{1}{(2\pi)^n} \int_{\mathbb{R}^{n}} \exp{\left(-(t-t_0)\vert\bm{\xi}\vert^{2}\right)}\exp{\left(\complexi\langle\bm{x}-\bm{y},\bm{\xi}\rangle\right)}\differential\bm{\xi}\nonumber \\ 
    &= \frac{1}{(2\pi)^n} \!\prod_{k=1}^{n}\!\left(\int_{-\infty}^{\infty}\!\!\!\exp{\left(-\xi_{k}^{2}(t-t_0) + \complexi(x_{k}-y_{k})\xi_{k}\right)}\differential\xi_{k}\!\right). \label{HeatKernelFirstStep}
\end{align}
To evaluate \eqref{HeatKernelFirstStep}, we specialize the Gaussian-like integral\footnote{This integral is simply the Fourier transform of the Gaussian with appropriate scaling.}
\begin{align}
    \int_{-\infty}^{\infty}\!\!\exp{\left(\!-\frac{1}{2}ax^2 + \complexi J x\!\right)} \differential x = \left(\frac{2\pi}{a}\right)^{\!1/2} \!\!\exp{\left(\!-\frac{J^2}{2a}\!\right)}
    \label{GenInt}
\end{align}
with $a = 2(t-t_0)$, $J = x_k - y_k$. This yields
\begin{align*}
    &\kappa_{\mathrm{heat}}(t_0, \bm{x}, t, \bm{y}) \nonumber\\
    &= \frac{1}{(2\pi)^n} \prod_{k=1}^{n} \left(\frac{\pi}{t-t_0}\right)^{1/2} \exp{\left( - \frac{(x_k - y_k)^2}{4(t-t_0)}\right)} \\ &= \frac{1}{(4\pi (t-t_0))^{n/2}}\exp{\left(-\frac{\vert \bm{x} - \bm{y} \vert^2}{4(t-t_0)}\right)},
\end{align*}
which is indeed the well-known heat kernel \cite[p. 44-47]{evans2022partial}.
\end{example}


In summary, we may derive the solution of a PDE \eqref{PDE_form} in the kernel form \eqref{kernel_form} by first identifying the associated Weyl operator $H(\bm{X}, \bm{D}) = \exp{\left( -(t-t_0) \mathcal{L}(\bm{X}, \bm{D})\right)}$ as in Sec. \ref{subsec:WeylOperator}. This is then followed up with a computation of the corresponding Weyl symbol $h(\bm{x}, \bm{\xi})$ as in Sec. \ref{subsec:WeylSymbol}. Using this Weyl symbol $h$, we then arrive at an expression for the kernel $\kappa$ via \eqref{SymbolToKernel}.

For pedagogical clarity, our Examples 2-4 so far illustrate the aforesaid process for the classical heat kernel. Now the question arises: can these computation be performed for reaction-diffusion PDEs with quadratic state-dependent reaction rate that appears in the associated Schr\"{o}dinger bridge problem with quadratic state cost?

It turns out that for such problems, the Weyl symbol is most conveniently determined via a PDE approach \cite[p. 130-132]{taylor1984noncommutative} that we detail in Sec. \ref{sec:ClassicalSBFixedQ}. To prepare ground for this PDE approach, we need the product rule of Weyl calculus, explained next. 

\subsection{Product Rule of Weyl Calculus}\label{subsec:ProductRuleWeylCalculus}
With the notations from Sec. \ref{subsec:WeylOperator}, consider a Weyl operator $C(\bm{X}, \bm{D})$ which admits a decomposition 
\begin{align}
    C(\bm{X}, \bm{D}) = A(\bm{X}, \bm{D}) B(\bm{X}, \bm{D})
    \label{OperatorDecomposn}
\end{align}
in terms of two other Weyl operators $A,B$. Consider the corresponding Weyl operator-Weyl symbol pairs $\left(A,a\right), \left(B,b\right)$ and $\left(C,c\right)$. Given the operator decomposition \eqref{OperatorDecomposn}, the product rule is a recipe to determine the Weyl symbol $c$ from the Weyl symbols $a,b$.

It is known that \cite[p. 161]{hormander2018} if either $A$ or $B$ is a polynomial, then the Weyl symbol $c(\bm{x}, \bm{\xi})$ can be exactly determined via Taylor expansion of\footnote{Here $\bm{D}_{\bm{x}}:=\begin{pmatrix}D_{x_{1}}&\hdots& D_{x_{n}}\end{pmatrix}^{\!\top}$, and likewise for $\bm{D}_{\bm{\xi}},\bm{D}_{\bm{y}}, \bm{D}_{\bm{\eta}}$.}
\begin{align*}
    \exp{\left( \frac{\mathrm{i}}{2} \sigma\left( \bm{D}_{\bm{x}}, \bm{D}_{\bm{\xi}}; \bm{D}_{\bm{y}}, \bm{D}_{\bm{\eta}}\right)\right)} (a(\bm{x}, \bm{\xi}) b(\bm{y}, \bm{\eta}) )\big\vert_{\bm{y} = \bm{x}, \bm{\eta} =\bm{\xi}}
\end{align*}
with \emph{finitely many terms}, where $\sigma$ is the standard symplectic form \eqref{defSymplecticForm}. In this case, we can express $c(\bm{x}, \bm{\xi})$ as \cite[p. 128-129]{taylor1984noncommutative}
\begin{align}
    c(\bm{x}, \bm{\xi}) = \sum_{j=0}^{d_{A}\wedge d_{B}} \frac{1}{j!} \{ a, b\}_{j}(\bm{x}, \bm{\xi}),
    \label{ProductRuleEitherPoly}
\end{align}
where $\{ a, b\}_{j}(\bm{x},\bm{\xi})$ is defined as in \eqref{BracketDef},
and $d_{A}\wedge d_{B}$ is the minimum of the degrees $d_{A},d_{B}$ for the polynomials $A,B$. When only one of $A,B$ is a polynomial, the corresponding polynomial degree is to be used as the upper limit of the summation index $j$ in \eqref{ProductRuleEitherPoly}. The product rule \eqref{ProductRuleEitherPoly} will be useful in our calculations in Sec. \ref{subsec:deriveWeylSymbolPDE}.

When neither $A$ nor $B$ is polynomial, the equality in \eqref{ProductRuleEitherPoly} is replaced with asymptotic equivalence:
$c(\bm{x}, \bm{\xi}) \sim \sum_{j\in \mathbb{N}_{0}} \frac{1}{j!} \{ a, b\}_{j}(\bm{x}, \bm{\xi})$.

\section{Kernel for the Schr\"{o}dinger Bridge with Quadratic State Cost}\label{sec:ClassicalSBFixedQ}
Our recent work \cite{teter2024schr} highlighted that solving a generic instance\footnote{i.e., with generic endpoint distributions which have finite second moments. The special case of Gaussian endpoints has been studied before in \cite[Sec. III]{chen2015optimal}.} of the Schr\"{o}dinger bridge problem with quadratic state cost leads to finding the Markov kernel of the reaction-diffusion PDE IVP
\begin{align}
\dfrac{\partial\widehat{\varphi}}{\partial t} &= \left(\Delta_{\bm{z}} -\frac{1}{2}\bm{z}^{\top} \bm{Q}\bm{z}\right)\widehat{\varphi}, \quad \widehat{\varphi}(t_0,\cdot)=\widehat{\varphi}_{0},
\label{OriginalPDE}
\end{align}
for given $\bm{Q}\succeq\bm{0}$, wherein $\bm{z}\in\mathbb{R}^{n}$ and the unknown $\widehat{\varphi}$ is a Schr\"{o}dinger factor. Here and throughout, $0\leq t_{0} \leq t < \infty$.

Problem \eqref{OriginalPDE} is indeed of the form \eqref{PDE_form}. The associated kernel in \eqref{kernel_form} was derived in \cite{teter2024schr} via tedious computation with Hermite polynomials. Since that computation is difficult to generalize for other state costs (thus for other state-dependent reaction rates), we explore applying the Weyl calculus ideas in Sec. \ref{sec:WeylCalculus} for finding such kernels. Here we show that the steps outlined in Sec. \ref{sec:WeylCalculus} indeed recovers the kernel for \eqref{OriginalPDE} in a systematic manner without invoking Hermite polynomials.

For $\bm{Q}\succeq\bm{0}$, consider the eigen-decomposition $\frac{1}{2}\bm{Q} = \bm{V}^{\top}\bm{\Lambda}\bm{V}$, where $\bm{\Lambda}$ is an $n \times n$ diagonal matrix with its main diagonal comprising the eigenvalues $\{\lambda_k\}_{k=1}^{n} \in \mathbb{R}_{\geq 0}^{n}$ of $\frac{1}{2}\bm{Q}$. Now consider a linear change-of-variable $\bm{z}\mapsto\bm{x}$ as
\begin{align}
\bm{x} := \bm{V}\bm{z}.
\label{SpectralChangeOfVariable}
\end{align}
Then,
\begin{align}
\widehat{\nu}(t,\bm{x}):=\widehat{\varphi}\left(t,\bm{z}=\bm{V}^{\top}\bm{x}\right),
\label{RelateNuhatPhihat}    
\end{align}
and $\Delta_{\bm{z}} \widehat{\varphi}(t,\bm{z}) = \Delta_{\bm{x}} \widehat{\nu}(t,\bm{x})$. Therefore, in the $\bm{x}$ coordinates, \eqref{OriginalPDE} takes the form
\begin{align}
    \frac{\partial \widehat{\nu}}{\partial t} &=  \Delta_{\bm{x}} \widehat{\nu} - \left(\bm{x}^{\top}\bm{\Lambda} \bm{x}\right)\widehat{\nu}  \nonumber  \\
    &=\displaystyle\sum_{k=1}^{n}\left(\dfrac{\partial^2}{\partial x_k^2} - \lambda_k x_k^2\right)\widehat{\nu}. \label{NewPDE}
\end{align}

\subsection{The Weyl Operator $H_{\bm{\Lambda}}$}\label{subsec:WeylOpHLambda}
For notational convenience, we define the Weyl operator
\begin{align}
    Q_{\bm{\Lambda}}(\bm{X}, \bm{D}) := \vert \bm{D} \vert^2 + \sum_{k = 1}^{n} \lambda_k X_k^2.
    \label{defQLambdaOperator}
\end{align}
Following Sec. \ref{subsec:WeylOperator}, the Weyl operator of \eqref{NewPDE} is
\begin{align}
    H_{\bm{\Lambda}}(\bm{X}, \bm{D}) = \exp{\left( - (t-t_0)Q_{\bm{\Lambda}}(\bm{X}, \bm{D}) \right)},
\label{WeylOperatorSBwithQuadraticStateCost}    
\end{align}
i.e., the Weyl operator of \eqref{NewPDE}, $H_{\bm{\Lambda}}(\bm{X}, \bm{D})$, can be seen as a composite operator.

Note that for the Weyl operator $Q_{\bm{\Lambda}}$ in \eqref{defQLambdaOperator}, we can readily find its Weyl symbol $q_{\bm{\Lambda}}$ using
\eqref{hSeriesRepresentation} as  
\begin{align}
    q_{\bm{\Lambda}}(\bm{x}, \bm{\xi}) = \vert \bm{\xi} \vert^2 + \sum_{k = 1}^{n} \lambda_k x_k^2.
    \label{q_Symbol}
\end{align}
However, for the composite Weyl operator $H_{\bm{\Lambda}}$, determining the corresponding Weyl symbol $h_{\bm{\Lambda}}$ via \eqref{hSeriesRepresentation} becomes challenging since partial derivatives of $H_{\bm{\Lambda}}(\bm{X}, \bm{D})$ of any order are nonzero. To circumvent this issue, we adopt a PDE approach from \cite[p. 130-132]{taylor1984noncommutative} explained next. The main idea is to derive a PDE IVP for $h_{\bm{\Lambda}}$ (Sec. \ref{subsec:deriveWeylSymbolPDE}) and to solve the same in closed form (Sec. \ref{subsec:solveWeylSymbolPDE}).

\subsection{Deriving the PDE for the Weyl Symbol $h_{\bm{\Lambda}}$}\label{subsec:deriveWeylSymbolPDE}
Since the Weyl operator $H_{\bm{\Lambda}}$ must satisfy \eqref{NewPDE}, we have
\begin{align*}
    \frac{\partial}{\partial t} H_{\bm{\Lambda}}(\bm{X}, \bm{D}) = -Q_{\bm{\Lambda}}(\bm{X}, \bm{D}) H_{\bm{\Lambda}}(\bm{X}, \bm{D}).
\end{align*}
Then, applying \eqref{ProductRuleEitherPoly}, we note that the Weyl symbol $h_{\bm{\Lambda}}$ satisfies
\begin{align}
    \frac{\partial}{\partial t} h_{\bm{\Lambda}}(\bm{x}, \bm{\xi}) &= - \sum_{j=0}^{2} \frac{1}{j!} \{ q_{\bm{\Lambda}}, h_{\bm{\Lambda}}\}_{j}(\bm{x}, \bm{\xi}).
\label{hLambdaPDE}
\end{align}
Applying \eqref{BracketDef} and \eqref{q_Symbol}, we find that 
\begin{align}
    \{q_{\bm{\Lambda}}, h_{\bm{\Lambda}}\}_1 = 0,
    \label{PoissonBracket1}
\end{align}
and
\begin{align}
     &\{ q_{\bm{\Lambda}}, h_{\bm{\Lambda}}\}_2 \nonumber\\
     &=\left( \sum_{k=1}^{n} \left(2\lambda_k(t-t_0) - 2\lambda_k^2(t-t_0)^2x_k^2 - 2\lambda_k(t-t_0)^2\xi_k^2\right) \!\!\right)\nonumber\\ & \hspace{2cm}  \times \exp{\left(-(t-t_0)\left( \vert \bm{\xi} \vert^2 + \sum_{k = 1}^{n} \lambda_k x_k^2 \right) \right)} \nonumber\\ 
     &= \left( 2 (t-t_0) \; \mathrm{trace}(\bm{\Lambda}) - 2(t-t_0)^2 \left(\bm{x}^{\top}\bm{\Lambda}^{2}\bm{x} +  \bm{\xi}^{\top}\bm{\Lambda}\bm{\xi} \right)\right)\nonumber\\ 
     & \hspace{2cm} \times \exp{\left(-(t-t_0) q_{\bm{\Lambda}}(\bm{x}, \bm{\xi}) \right)}.
     \label{PoissonBracket2}
\end{align}
Since the second order derivative of $h_{\bm{\Lambda}}(\bm{x}, \bm{\xi})$ w.r.t. $x_k \forall\: k\in[n]$, denoted as $\partial^{2}_{x_k} h(x, \xi)$, equals
\begin{align*}
    \left(-2\lambda_k(t-t_0) + 4\lambda_k^2x_k^2 (t-t_0)^2\right) \exp{\left(-(t-t_0)q_{\bm{\Lambda}}(\bm{x}, \bm{\xi})\right)}, 
\end{align*}
and likewise $\partial^{2}_{\xi_k} h_{\bm{\Lambda}}(\bm{x}, \bm{\xi})$ equals
\begin{align*}
    \left(-2(t-t_0) + 4\xi_k^2 (t-t_0)^2\right) \exp{\left(-(t-t_0) q_{\bm{\Lambda}}(\bm{x}, \bm{\xi})\right)}, 
\end{align*}
hence \eqref{PoissonBracket2} can be expressed as
\begin{align}
    \{ q_{\bm{\Lambda}}, h_{\bm{\Lambda}}\}_2 = -\frac{1}{2} \left( \sum_{k=1}^n d_k \partial^{2}_{\xi_k} h_{\bm{\Lambda}}(\bm{x}, \bm{\xi}) + \partial^{2}_{x_k} h_{\bm{\Lambda}}(\bm{x}, \bm{\xi}) \right).
    \label{PoissonBracket2Simplified}
\end{align}
Combining \eqref{hLambdaPDE}, \eqref{PoissonBracket1} and \eqref{PoissonBracket2Simplified}, we then have
\begin{align}
    \frac{\partial}{\partial t} h_{\bm{\Lambda}}(\bm{x}, \bm{\xi}) &= -q_{\bm{\Lambda}}(\bm{x}, \bm{\xi}) h_{\bm{\Lambda}}(\bm{x}, \bm{\xi}) - \frac{1}{2} \{q_{\bm{\Lambda}}, h_{\bm{\Lambda}}\}_{2} (\bm{x}, \bm{\xi}) \nonumber \\
    &= -q_{\bm{\Lambda}} h_{\bm{\Lambda}} + \frac{1}{4} \left(  \sum_{k=1}^n \lambda_k \partial^{2}_{\xi_k} h_{\bm{\Lambda}} + \partial^{2}_{x_k} h_{\bm{\Lambda}} \right). \label{PDEinh}
\end{align}
We next focus on solving for $h_{\bm{\Lambda}}$ from \eqref{PDEinh} subject to the initial condition $h_{\bm{\Lambda}}\vert_{t=t_{0}}=1$.

\subsection{Solving the PDE for the Weyl Symbol $h_{\bm{\Lambda}}$}\label{subsec:solveWeylSymbolPDE}
We note that the Weyl symbol $q_{\bm{\Lambda}}$ in \eqref{q_Symbol} is time-invariant. Let $q_k := \lambda_k x_k^2 + \xi_k^2$ $\forall k\in[n]$, so that $q_{\bm{\Lambda}} = \sum_{k=1}^{n} q_k$.

We seek to solve \eqref{PDEinh} in the form
\begin{align}
h_{\bm{\Lambda}}(\bm{x}, \bm{\xi}) = g(t, q_1, \hdots, q_n).
\label{hAnsatz}    
\end{align}
Substituting \eqref{hAnsatz} into \eqref{PDEinh} yields
\begin{align}
    \frac{\partial g}{\partial t} &= - \left(\sum_{k=1}^{n} q_k\right)g + \frac{1}{4} \sum_{k=1}^{n}\!\left(\! \frac{\partial }{\partial x_k} \!\left(\! \frac{\partial g}{\partial q_k} \frac{\partial q_k}{\partial x_k}\!\right)\!\right. \nonumber\\   &\left.\qquad\qquad\qquad\qquad\qquad+ \lambda_k \frac{\partial }{\partial \xi_k} \!\left(\! \frac{\partial g}{\partial q_k} \frac{\partial q_k}{\partial \xi_k}\!\right)\!\!\right). 
    \label{gPDE}
\end{align}
As
\begin{align*}
    \frac{\partial }{\partial x_k} \left( \frac{\partial g}{\partial q_k} \frac{\partial q_k}{\partial x_k}\right) &= \frac{\partial }{\partial x_k} \left( \frac{\partial g}{\partial q_k} \left(2\lambda_{k} x_{k}\right)\right) \\ &= \frac{\partial^2 g}{\partial q_k^2} \left(4\lambda_k^2x_{k}^{2}\right) + 2\lambda_k \frac{\partial g}{\partial q_k},
\end{align*}
and likewise
\begin{align*}
    \frac{\partial }{\partial \xi_k} \left( \frac{\partial g}{\partial q_k} \frac{\partial q_k}{\partial \xi_k}\right) &= \frac{\partial }{\partial \xi_k} \left( \frac{\partial g}{\partial q_k} (2\xi_k)\right) \\ &=  4\xi_k^2 \frac{\partial^2 g}{\partial q_k^2}  + 2\frac{\partial g}{\partial q_k},
\end{align*}
the PDE \eqref{gPDE} simplifies to
\begin{align}
    \frac{\partial g}{\partial t} = - \left(\sum_{k=1}^{n} q_k\right)g + \sum_{k=1}^{n} \left(\lambda_k q_k \frac{\partial^2 g}{\partial q_k^2} + \lambda_k \frac{\partial g}{\partial q_k}\right).
    \label{PDEinG}
\end{align}

To solve \eqref{PDEinG}, we consider the ansatz
\begin{align}
g(t, q_1, \hdots, q_n) = \alpha(t) \exp{\left(-\sum_{k=1}^{n}\beta_k(t)q_k\right)}
\label{SymbolGuess}
\end{align}
for suitably smooth $\alpha(t),\beta_{1}(t), \hdots, \beta_{n}(t)$ that are not identically zero for any $t_{0} < t$. Recall from Sec. \ref{subsec:Kernel} that $h_{\bm{\Lambda}}\vert_{t=t_{0}}=1$, and hence from \eqref{hAnsatz} and \eqref{SymbolGuess}, the initial conditions for $\alpha(\cdot),\beta_{1}(\cdot),\hdots,\beta_{n}(\cdot)$ are
\begin{align}
\alpha\left(t_{0}\right) = 1, \quad \beta_{k}\left(t_{0}\right) = 0 \quad\forall\:k\in[n].
\label{alphabetaIC}    
\end{align}

Next, notice that
\begin{subequations}
\begin{align}
    \frac{\partial g}{\partial t} &= \left( \dot{\alpha} - \alpha \sum_{k=1}^n \dot{\beta}_k q_k \right) \exp{\left(-\sum_{k=1}^{n}\beta_k(t)q_k\right)}, \\
    \frac{\partial g}{\partial q_k} &= -\alpha \beta_k \exp{\left(-\sum_{k=1}^{n}\beta_k q_k\right)}, \\ \frac{\partial^2 g}{\partial q_k^2} &= \alpha \beta_k^2 \exp{\left(-\sum_{k=1}^{n}\beta_k q_k\right)},
\end{align}
\label{DerivativesOfg}    
\end{subequations}
where $\dot{\alpha}:=\frac{\differential \alpha}{\differential t},\dot{\beta}:=\frac{\differential \beta}{\differential t}$. Substituting \eqref{SymbolGuess} into \eqref{PDEinG}, using \eqref{DerivativesOfg}, and dividing through by $\exp{\left(-\sum_{k=1}^{n}\beta_k(t)q_k\right)}$ yields
\begin{align*}
\dot{\alpha} - \alpha \sum_{k=1}^n \dot{\beta}_k q_k = \sum_{k=1}^{n} \left(- \alpha  q_k + \lambda_k \alpha \beta_k^2 q_k  + \lambda_k\left( -\alpha \beta_k \right) \right).
\end{align*}
Rearranging the above to collect terms involving $q_k$ in the RHS, we obtain
\begin{align}
    &\dot{\alpha} + \sum_{k=1}^{n}\lambda_k \alpha \beta_k = \sum_{k=1}^{n} \left(\alpha \dot{\beta}_k - \alpha  + \lambda_k \alpha \beta_k^2 \right)q_k.
    \label{LastPDE}
\end{align}
As the LHS of \eqref{LastPDE} is independent of $q_k \; \forall k\in [n]$, we note that the coefficients of $q_{k}$ in the RHS must be zero $\forall k\in [n]$. This observation, together with the fact that $\alpha(t)\neq 0$ (otherwise \eqref{SymbolGuess} gives trivial solution), results in the nonlinear ODE IVPs
\begin{align}
    \dot{\beta}_k =  1 - \lambda_k \beta_k^{2}, \quad \beta_{k}\left(t_{0}\right) = 0 \quad \forall k \in [n],
    \label{BetaODE}
\end{align}
where we have imposed the $\beta_{k}(\cdot)$ initial conditions from \eqref{alphabetaIC}.

Separation-of-variables and $u$ substitution with $\tanh(u) = \sqrt{\lambda_k}\beta_k$, determine the solution to \eqref{BetaODE} as
\begin{align}
    \beta_k(t) = \frac{1}{\sqrt{\lambda_k}} \tanh\left(\sqrt{\lambda_k}(t-t_0)\right), \quad \forall k\in[n].
    \label{BetaSol}
\end{align}
Since all coefficients in the RHS of \eqref{LastPDE} are zero, combining \eqref{LastPDE} and \eqref{BetaSol} gives
\begin{align}
    \frac{\dot{\alpha}}{\alpha} = - \sum_{k=1}^{n}\sqrt{\lambda_k} \tanh\left(\sqrt{\lambda_k}(t-t_0)\right), \quad \alpha\left(t_{0}\right) = 1,
    \label{alphaODE}
\end{align}
where we have imposed the $\alpha(\cdot)$ initial condition from \eqref{alphabetaIC}.
Direct integration yields the solution of the ODE IVP \eqref{alphaODE}:
\begin{align}
    \alpha(t) = \prod_{k=1}^{n} \frac{1}{\cosh\left( \sqrt{\lambda_k}(t - t_0)\right)}.
    \label{alphaSoln}
\end{align}
Combining \eqref{hAnsatz}, \eqref{SymbolGuess}, \eqref{BetaSol}, \eqref{alphaSoln}, we have thus determined that the Weyl symbol $h_{\bm{\Lambda}}$ associated with the Weyl operator $H_{\bm{\Lambda}}$, is 
\begin{align}
    h_{\bm{\Lambda}}(\bm{x}, \bm{\xi}) &=  \left(\prod_{k=1}^{n} \frac{1}{\cosh\left( \sqrt{\lambda_k}(t - t_0)\right)}\right) \nonumber \\ & \times \exp\!{ \left( -\sum_{k=1}^{n}\frac{ \lambda_k x_k^2 + \xi_{k}^2}{\sqrt{\lambda_k}} \tanh\left(\sqrt{\lambda_k}(t-t_0)\right)\! \right)}. 
    \label{SymbolQ}
\end{align}

\subsection{The Kernel $\kappa_{\bm{\Lambda}}$}\label{subsec:kappaLambda}
Following Sec. \ref{subsec:Kernel}, we now apply \eqref{SymbolToKernel} to determine the kernel $\kappa_{\bm{\Lambda}}$ corresponding to the Weyl symbol $h_{\bm{\Lambda}}$ given by \eqref{SymbolQ}.

Since \eqref{SymbolQ} is a product of $n$ terms, each with its unique subscript $k\in [n]$, so \eqref{SymbolToKernel} can be written as the product of $n$ univariate integrals, each being
\begin{align}
    &\frac{1}{2\pi}\frac{1}{\cosh\left( \sqrt{\lambda_k}(t - t_0)\right)} \nonumber\\
    &\times \exp{\left( -\frac{\sqrt{\lambda_k} \left(x_k^2 + 2x_k y_k+ y_k^2\right)}{4} \tanh\left( \sqrt{\lambda_k}(t-t_0)\right) \right)} \nonumber \\ & \times \!\!\underbrace{\int_{-\infty}^{+\infty} \!\!\!\!\!\exp{\!\left(\! -\frac{\xi_k^2}{\sqrt{\lambda_k}} \tanh\!\left(\!\sqrt{\lambda_k}(t-t_0)\!\right)\!\! + \!\complexi(x_k-y_k)\xi_k\!\!\right)} \differential\xi_k}_{I_{k}}, \label{FactoredKernelInt}
\end{align}
where $k\in[n]$.

Invoking \eqref{GenInt} with $$a = \frac{2}{\sqrt{\lambda_k}}\tanh\left( \sqrt{\lambda_k}(t - t_0)\right), \quad J = x_k - y_k,$$ 
the integral $I_{k}$ in \eqref{FactoredKernelInt} evaluates to
\begin{align*}
    \left( \frac{\pi\sqrt{\lambda_k}}{\tanh\left( \sqrt{\lambda_k}(t - t_0)\right)} \right)^{\!\!1/2} \!\exp{\left(\!-\frac{\sqrt{\lambda_k}(x_k - y_k)^2}{4\tanh\left( \sqrt{\lambda_k}(t - t_0)\!\right)}\right)}.
\end{align*}
Thus, \eqref{FactoredKernelInt} is the product of a pre-factor and an exponential term, wherein the pre-factor is
\begin{align}
    &\frac{1}{2\pi}\frac{1}{\cosh\left( \sqrt{\lambda_k}(t - t_0)\right)} \left( \frac{\pi\sqrt{\lambda_k}}{\tanh\left( \sqrt{\lambda_k}(t - t_0)\right)} \right)^{1/2} \nonumber\\ &= \frac{\lambda_k^{1/4}}{\sqrt{2\pi \sinh{(2\sqrt{\lambda_k}(t-t_0)})}},
    \label{Prefactor}
\end{align}
and letting $\theta_{k}:=\sqrt{\lambda_k}(t - t_0)$, the argument of the exponential term is
\begin{align}
    &- \frac{\sqrt{\lambda_k} (x_k^2 + 2x_k y_k + y_k^2)}{4} \tanh\theta_{k} -\frac{\sqrt{\lambda_k}(x_k - y_k)^2}{4\tanh\theta_{k}}\nonumber\\ 
    &=-\frac{\sqrt{\lambda_{k}}}{4}\bigg\{\left(x_k^2 + y_k^2\right)\left(\frac{\sinh\theta_{k}}{\cosh\theta_{k}}+\frac{\cosh\theta_{k}}{\sinh\theta_{k}}\right)\nonumber\\
&\qquad\qquad\quad\;\;\,+2x_{k}y_{k}\left(\frac{\sinh\theta_{k}}{\cosh\theta_{k}}-\frac{\cosh\theta_{k}}{\sinh\theta_{k}}\right)\bigg\}\nonumber\\
    &=-\frac{\sqrt{\lambda_k}}{2} (x_k^2 + y_k^2) \frac{\cosh\left(2\theta_{k}\right)}{\sinh\left(2\theta_{k}\right)} + \frac{\sqrt{\lambda_k}x_k y_k}{\sinh\left(2\theta_{k}\right)}.
    \label{ArgumentOfExp}
\end{align}
Therefore, taking the product of the $n$ integrals of the form  \eqref{FactoredKernelInt} from $k=1$ to $k=n$, and substituting back $\theta_{k}=\sqrt{\lambda_k}(t - t_0)$, we arrive at the kernel
\begin{align}
    &\kappa_{\bm{\Lambda}}(t_0, \bm{x}, t, \bm{y})\nonumber\\
    &=  \left(\prod_{k=1}^{n} \frac{\lambda_k^{1/4}}{\sqrt{2\pi \sinh{(2\sqrt{\lambda_k}(t-t_0)})}} \right) \nonumber\\ &\times \exp{\left( -
    \sum_{k=1}^{n}\frac{\sqrt{\lambda_k}}{2} (x_k^2 + y_k^2) \frac{\cosh{(2\sqrt{\lambda_k}(t-t_0))}}{\sinh{(2\sqrt{\lambda_k}(t-t_0))}} \right)} \nonumber\\ & \times \exp{\left( \sum_{k=1}^{n}\sqrt{\lambda_k}x_ky_k \left(\frac{1}{\sinh{(2\sqrt{\lambda_k}(t-t_0))}} \right) \right)}.
    \label{kappaLambda}
\end{align}

The expression \eqref{kappaLambda} for the kernel matches with that derived in \cite{teter2024schr}. Herein, by leveraging the Weyl calculus, we circumvent the use of Hermite polynomial-based tedious computation in \cite{teter2024schr}.

\begin{remark}
The kernel $\kappa_{\bm{\Lambda}}$ in \eqref{kappaLambda}, when specialized to the case $\frac{1}{2}\bm{Q}=\bm{\Lambda}=\bm{I}_{n}$, is known as the Mehler kernel \cite{mehler1866ueber}, \cite[Thm. 1]{robert2021coherent} in the quantum mechanics literature where it appears in the solution of the time-dependent Schr\"{o}dinger equation for isotropic quantum harmonic oscillator. There is a substantial literature generalizing this result for time-independent non-self-adjoint \cite{hormander1995symplectic}, time-dependent self-adjoint \cite{mehlig2001semiclassical,combescure2005quadratic,gosson2005weyl}, and time-dependent non-self-adjoint \cite{pravda2018generalized} quadratic Hamiltonians. 
\end{remark}

\begin{remark}\label{Reamrk:sim}
For numerical simulation results solving the SB problem with quadratic state cost using the kernel $\kappa_{\bm{\Lambda}}$, we refer the readers to \cite[Figs. 1 and 5]{teter2024schr}.
\end{remark}

\subsection{Generalization}\label{subsec:generalization}
Formula \eqref{kappaLambda} serves as the kernel or Green's function for \eqref{NewPDE}, and helps write the solution for the PDE IVP \eqref{OriginalPDE} as
\begin{align}
\widehat{\varphi}(t,\bm{z}) = \widehat{\nu}\left(t,\bm{x}\right)\big\vert_{\bm{x}=\bm{Vz}} \!\!=\!\! \int_{\mathbb{R}^{n}}\!\!\kappa_{\bm{\Lambda}}\!\left(t_{0},\bm{Vz},t,\bm{y}\right)\varphi_{0}\!\left(\bm{V}^{\top}\bm{y}\right)\differential\bm{y}.
\label{OrginalPDESoln}
\end{align}
We now extend this result for a modified version of \eqref{OriginalPDE} given by
\begin{align}
\dfrac{\partial\widehat{\varphi}}{\partial t} &= \left(\Delta_{\bm{z}} -\left(\frac{1}{2}\bm{z}^{\top} \bm{Q}\bm{z}+ \boldsymbol{r}^{\top} \boldsymbol{z} + s\right)\right)\widehat{\varphi}, \quad \widehat{\varphi}(t_0,\cdot)=\widehat{\varphi}_{0},
\label{ModifiedOriginalPDE}
\end{align}
where $\bm{Q}\succeq\bm{0}$, $\bm{r}\in\mathbb{R}^{n}$, $s\in\mathbb{R}$. In other words, we generalize the reaction rate to have additional affine terms, i.e., \emph{generic convex quadratic reaction rate}:
$\frac{1}{2}\bm{z}^{\top} \bm{Q}\bm{z}+ \boldsymbol{r}^{\top} \boldsymbol{z} + s$.

\begin{theorem}\label{ThmGenericQuadReactionRate}
Given $\bm{Q}\succeq\bm{0}$, $\bm{r}\in\mathbb{R}^{n}$, $s\in\mathbb{R}$, consider the eigen-decomposition $\frac{1}{2}\bm{Q} = \bm{V}^{\top}\bm{\Lambda}\bm{V}$, where the main diagonal entries of the diagonal matrix $\Lambda$ are $\{\lambda_k\}_{k=1}^{n}$. Let $\bm{v}_{k}$ be the $k$th column of $\bm{V}^{\top}$ for all $k\in[n]$. 

The solution for the PDE IVP \eqref{ModifiedOriginalPDE} is 
\begin{align}
\widehat{\varphi}(t,\bm{z}) =\int_{\mathbb{R}^{n}}\!\!\kappa_{\left(\bm{\Lambda},\bm{r},s\right)}\!\left(t_{0},\bm{Vz},t,\bm{y}\right)\varphi_{0}\!\left(\bm{V}^{\top}\bm{y}\right)\differential\bm{y},
\label{ModifiedOrginalPDESoln}
\end{align}
where the kernel
\begin{align}
    &\kappa_{\left(\bm{\Lambda},\bm{r},s\right)}(t_0, \bm{x}, t, \bm{y})\nonumber\\
    &=  \left(\prod_{k=1}^{n} \frac{\lambda_k^{1/4}\exp{(-c_k(t-t_0))}}{\sqrt{2\pi \sinh{(2\sqrt{\lambda_k}(t-t_0)})}} \right) \nonumber\\ &\times \exp\left(
    \sum_{k=1}^{n}-\frac{\sqrt{\lambda_k}}{2} (x_k^2 + y_k^2) \frac{\cosh{(2\sqrt{\lambda_k}(t-t_0))}}{\sinh{(2\sqrt{\lambda_k}(t-t_0))}} \right. \nonumber\\ & \left. \qquad\qquad +\frac{\sqrt{\lambda_k}x_ky_k}{\sinh{(2\sqrt{\lambda_k}(t-t_0))}} \right.\nonumber\\
    &\left. -\dfrac{\left(\frac{1}{2}\bm{r}^{\top}\bm{v}_{k}\left(x_k + y_k\right) + \frac{1}{4\lambda_{k}}\left(\bm{r}^{\top}\bm{v}_{k}\right)^{2}\right)\!\tanh(\sqrt{\lambda_k}(t-t_0))}{\sqrt{\lambda_{k}}}\!\!\right),
\label{Modifiedkappa}
\end{align}
and the constants 
\begin{align}
c_k := \frac{1}{4\lambda_k}(\bm{r}^{\top}\bm{v}_k)^2 - \frac{s}{n}, \quad \forall k\in[n].
\label{defck}
\end{align}
\end{theorem}
\begin{proof}
We proceed as before by letting $\bm{x} :=\bm{Vz}$, $\widehat{\nu}(t, \bm{x}) := \widehat{\varphi}(t, \boldsymbol{z} = \boldsymbol{V}^{\top}\boldsymbol{x})$, to get 
$$\frac{\partial \widehat{\nu}}{\partial t} = \left( \Delta_{\boldsymbol{x}} - \sum_{k=1}^{n} \left( \lambda_k x_k^2 + \boldsymbol{r}^{\top}\boldsymbol{v}_k x_k + \frac{1}{n}s \right)\right) \widehat{\nu}.$$
The formulae \eqref{defQLambdaOperator}, \eqref{WeylOperatorSBwithQuadraticStateCost} and \eqref{q_Symbol} now become
\begin{align*}
Q_{\left(\bm{\Lambda},\bm{r},s\right)}(\bm{X}, \bm{D}) &= s + \vert \bm{D} \vert^2 + \sum_{k = 1}^{n} \left( \lambda_k X_k^2 + \boldsymbol{r}^{\top}\bm{v}_k X_k \right),\\
H_{\left(\bm{\Lambda},\bm{r},s\right)}(\bm{X}, \bm{D}) &=\exp{\left( - (t-t_0)Q_{\left(\bm{\Lambda},\bm{r},s\right)}(\bm{X}, \bm{D}) \right)},\\
q_{\left(\bm{\Lambda},\bm{r},s\right)}(\bm{x}, \bm{\xi}) &= s + \vert \bm{\xi} \vert^2 + \sum_{k = 1}^{n} \left( \lambda_k x_k^2 + \boldsymbol{r}^{\top}\bm{v}_k x_k \right).
\end{align*}
The associated Weyl symbol $h_{\left(\bm{\Lambda},\bm{r},s\right)}$ solves an analogue of \eqref{hLambdaPDE}, given by
$$\frac{\partial}{\partial t} h_{\left(\bm{\Lambda},\bm{r},s\right)}(\bm{x}, \bm{\xi}) = - \sum_{j=0}^{2} \frac{1}{j!} \{ q_{\left(\bm{\Lambda},\bm{r},s\right)}, h_{\left(\bm{\Lambda},\bm{r},s\right)}\}_{j}(\bm{x}, \bm{\xi}),$$
from which, direct calculation of the Poisson brackets again yields \eqref{PDEinh} with the initial condition $h_{\left(\bm{\Lambda},\bm{r},s\right)}\vert_{t=t_{0}}=1$.

Now let
\begin{align}
q_{k}:=\xi_k^2 + \lambda_k x_k^2 + \boldsymbol{r}^{\top}\boldsymbol{v}_k x_k + \frac{1}{n}s, \quad \forall k \in [n],
\label{DefNewqk}
\end{align}
so that $\sum_{k=1}^{n}q_{k}=q_{\left(\bm{\Lambda},\bm{r},s\right)}$. Following \eqref{hAnsatz}, we define the function $g$ via  
\begin{align}
h_{\left(\bm{\Lambda},\bm{r},s\right)}(\bm{x}, \bm{\xi})=g(t, q_1, \hdots, q_n),
\label{defNewg}    
\end{align}
and using \eqref{PDEinh}, arrive at the PDE \eqref{gPDE}.

Using the definition \eqref{DefNewqk}, we next simplify \eqref{gPDE} as
\begin{align}
    \frac{\partial g}{\partial t} &= \sum_{k=1}^{n} \left(-q_k g + \lambda_k \frac{\partial g}{\partial q_k} + \frac{\partial^2 g}{\partial q_k^2} \left( \lambda_k\left(q_k + c_k \right)\right)\right).
    \label{NewgPDESimplified}
\end{align}
The PDE \eqref{NewgPDESimplified} generalizes \eqref{PDEinG}.

To solve \eqref{NewgPDESimplified}, we use the anstaz \eqref{SymbolGuess} with initial conditions \eqref{alphabetaIC}. Following the same steps as before, we again find that $\beta_{k}(t)$ are given by \eqref{BetaSol}. However, $\alpha(t)$ now solves a modified version of \eqref{alphaODE}, given by 
\begin{align}
    \frac{\dot{\alpha}}{\alpha} &= - \sum_{k=1}^{n}\bigg\{\sqrt{\lambda_k} \tanh\left(\sqrt{\lambda_k}(t-t_0)\right) \nonumber\\
    &\qquad+ c_k \tanh^2\left(\sqrt{\lambda_k}(t-t_0)\right)\bigg\}, \quad \alpha\left(t_{0}\right) = 1.
    \label{NewalphaODE}
\end{align}
By direct integration, the solution for \eqref{NewalphaODE} is
\begin{align}
&\alpha(t) = \left(\prod_{k=1}^{n}\frac{1}{\textrm{cosh}\left(\sqrt{\lambda_k} (t-t_0)\right)} \right) \times\nonumber\\
&\exp{ \left(\sum_{k=1}^{n}\left( -c_k (t-t_0) + \frac{c_k}{\sqrt{\lambda_k}}\tanh{\sqrt{\lambda_k}(t-t_0)} \right) \right)},
\label{NewalphaSoln}    
\end{align}
which generalizes the earlier formula \eqref{alphaSoln}. 

Combining these $\alpha(t),\beta_{k}(t)$ together with \eqref{SymbolGuess} and \eqref{defNewg}, we obtain the Weyl symbol 
\begin{align}
    &h_{\left(\bm{\Lambda},\bm{r},s\right)}(\bm{x}, \bm{\xi}) =  \left(\prod_{k=1}^{n}\frac{1}{\cosh\left( \sqrt{\lambda_k}(t - t_0)\right)}\right)\nonumber\\
    &\qquad\qquad\qquad\quad\times \exp{\left(-\sum_{k=1}^{n}c_k (t-t_0)\right)} \nonumber\\
    & \times \exp\left(-\sum_{k=1}^{n} \frac{\lambda_k x_k^2 + \xi_{k}^2 + \bm{r}^{\top}\bm{v}_k x_k + (\bm{r}^{\top}\bm{v}_k)^2/\left(4\lambda_k\right)} {\sqrt{\lambda_k}}\right.\nonumber\\    &\qquad\qquad\qquad\times\left.\tanh\left(\sqrt{\lambda_k}(t-t_0)\right)\right),
    \label{NewSymbolQ}
\end{align}
which generalizes the earlier derived \eqref{SymbolQ}.

To derive the kernel $\kappa_{\left(\bm{\Lambda},\bm{r},s\right)}$ from the Weyl symbol $h_{\left(\bm{\Lambda},\bm{r},s\right)}$, we next apply \eqref{SymbolToKernel} to \eqref{NewSymbolQ}. Following the computation similar to Sec. \ref{subsec:kappaLambda}, we arrive at \eqref{Modifiedkappa}. Notice that for $\bm{r}=\bm{0}, s=0$, the kernel \eqref{Modifiedkappa} reduces to \eqref{kappaLambda}, as expected. 
\end{proof}    




\section{Conclusions}\label{sec:conclusions}
We explained how ideas from Weyl calculus can be useful to derive the kernels associated with linear reaction-diffusion PDEs with state dependent reaction rates. Our primary motivation behind solving such PDEs came from the Schr\"{o}dinger bridge problems with additive state costs. These state costs regularize the classical minimum effort Lagrangian, and encourage the optimally controlled stochastic states to stay close to a desired level at all times while satisfying the endpoint distributional constraints. Numerically solving such problems with generic endpoint distributional data via dynamic Sinkhorn recursions, however, requires the Markov kernel associated with the corresponding reaction-diffusion PDE. This is where the Weyl calculus tools can be effective. We outlined the general computation steps:
$${\text{PDE}} \:\longrightarrow \:{\text{Weyl operator}} \:\longrightarrow \:{\text{Weyl symbol}} \:\longrightarrow \:{\text{kernel}},$$
and worked out the details for convex quadratic state cost. 

Our calculations here recovered the recent result in \cite{teter2024schr} where computations were not only more tedious but were also 
difficult to generalize for other variants of such PDEs. 
While Weyl calculus tools were invented to address problems in quantum mechanics, our contribution here is to highlight their efficacy in explicitly recovering the kernels of PDEs arising from diffusion and control problems. 
We hope this will foster broader appreciation of Weyl calculus in systems-control and motivate further applications. 

\bibliographystyle{IEEEtran}
\bibliography{references.bib}

\end{document}